\documentclass[12pt]{amsart}
\usepackage{amsfonts, amssymb, amsmath, amsthm, euscript}
\theoremstyle{plain}
\newtheorem{thm}{Theorem}[section]

\theoremstyle{definition}

\newtheorem*{Ack}{Acknowledgement}

\theoremstyle{remark}

\newtheorem{note}[thm]{}

\def\grRp{{\operatorname{gr}}{R^+}}
\def\grR{{\operatorname{gr}}{R}}
\def\ha{\widehat{a}}
\def\hb{\widehat{b}}

\def\Z{\mathbb{Z}}

\begin{document}

\title[Characterizations of finite dimensionality]{On graded characterizations of 
  finite \\ dimensionality for algebraic algebras}

\author{Edward S. Letzter}

\address{Department of Mathematics\\
        Temple University\\
        Philadelphia, PA 19122-6094}
      
      \email{letzter@temple.edu }

\date{\today}



\begin{abstract} We observe that a finitely generated algebraic
  algebra $R$ (over a field) is finite dimensional if and only if the
  associated graded ring $\grR$ is right noetherian, if and only if
  $\grR$ has right Krull dimension, if and only if $\grR$ satisfies a
  polynomial identity.
\end{abstract}

\maketitle


\section{Introduction}

Examples of infinite dimensional, finitely generated, algebraic
algebras (over fields) were first produced by Golod and Shafarevich in
1964 \cite{Gol-Saf}, providing a negative answer to the longstanding
and famous Kurosh Problem. Since the early 2000s there has been
increased interest -- and several new significant results -- in the
study of Kurosh-type and related problems for associative algebras;
see, e.g., \cite{Bel-Sma-Smo}, \cite{Smo2}, and \cite{Zel} for an
introduction and overview. A thumbnail sketch of relevant developments
could include: Smoktunowicz's 2002 construction of a simple nil ring
over an arbitrary countable field \cite{Smo1}; Bell and Small's 2002
construction of a finitely generated primitive algebraic algebra over
an arbitrary field \cite{Bel-Sma}; Lenagan and Smoktunowicz's 2007
construction of an infinite dimensional nil algebra, over an arbitrary
countable field, of finite Gelfand-Kirillov (GK-) dimension
\cite{Len-Smo}; Lenagan, Smoktunowicz, and Young's 2012 construction
of an infinite dimensional nil algebra, over an arbitrary countable
field, of GK-dimension at most three \cite{Len-Smo-You}; and Bell,
Small, and Smoktunowicz's construction of an infinite dimensional
primitive algebraic algebra, over an arbitrary countable field, of
GK-dimension at most six \cite{Bel-Sma-Smo}. 

Included among prior results specifically focused on associated graded
rings are Smoktunowicz's 2010 example of a finitely generated algebaic
algebra, over an arbitrary countable field, for which the associated
graded ideal spanned by homogeneous elements of positive degree is not
nil \cite{Smo3}, and Regev's 2010 theorem, for finitely generated
algebraic algebras over an uncountable field, that the associated
graded ideal spanned by homogeneous elements of positive degree must
be nil \cite{Reg}. In both cases, these ideals are graded nil. See
(\ref{Reg-Smo}) below.

Our main result, presented in somewhat more precise form in
(\ref{main-result}), asserts that a finitely generated algebraic
algebra $R$ over a field is finite dimensional if and only if the
associated graded ring $\grR$ is right noetherian, if and only if
$\grR$ has right Krull dimension, if and only if $\grR$ satisfies a
polynomial identity.

Numerous naturally arising examples of infinite dimensional, finitely
generated algebras whose associated graded rings are noetherian can be
found in \cite{Bro-Goo} and \cite{McC-Rob}.  In-depth treatments of
graded and filtered rings can be found in \cite{McC-Rob} and
\cite{Nas-VOy}. 

\begin{Ack}
I am grateful to Lance Small for helpful comments.
\end{Ack}

\section{Graded Characterizations of Finite Dimensionality} 

\begin{note} \label{setup} Let $k$ be a field, and let $R$ be a $k$-algebra
  generated by a finite dimensional $k$-subspace $V$ of $R$ with
  $1 \in V$.  The following standard setup will remain in effect
  throughout this note:

(i) Letting $V^{-1} = 0$ and $V^0 = k$, set $V^m$ equal to the
$k$-vector space spanned by products over $V$ of length $m$.

(ii) For $a \in R$, set $\vert a\vert$ equal to the minimum integer
$m$ such that $a \in V^m$, and set 
\[ \ha = a + V^m \in V^m/V^{m-1}.\]
Multiplication in the associated ($\Z$-)graded ring
\[ \grR = \bigoplus _{i=0}^\infty V^i/V^{i-1}\]
is determined via
\[ \ha \cdot \hb = ab + V^{m+n-1} \in V^{m+n}/V^{m+n-1}, \]
for $\vert a \vert = m$ and $\vert b \vert = n$. 

(iii) Set 
\[ \grRp = \bigoplus ^\infty _{i =1} V^i/V^{i-1} ,\]
a graded ideal of $\grR$ that is $k$-linearly spanned by the
homogeneous elements $\ha$, for $a \in R$, $|a| \geq 1$. It is easy to
check that $R$ is finite dimensional if and only if $\grR$ is finite
dimensional, if and only if $\grRp$ is nilpotent.

(iv) A subset $S$ of $R$ is \emph{nil} if each element of $S$ is
nilpotent. A graded subring of a group graded ring is \emph{graded
  nil} if each homogeneous element is nilpotent.
\end{note}

\begin{note} \label{Reg-Smo} Assuming $R$ is algebraic over $k$,
  Smoktunowicz proved that $\grRp$ need not be nil if $k$ is countable
  \cite{Smo3}, and A.~Regev proved that $\grRp$ must be nil if $k$ is
  uncountable \cite{Reg}. However, for arbitrary choices of $k$, it
  can easily be checked that $\grRp$ is graded nil if $R$ is algebraic
  over $k$.
\end{note}

\begin{note} \label{survey}
  Next, we give a brief survey of Jacobson's theorem \cite{Jac} (and some of its
  applications) concerning the nilpotence of certain subrings of
  artinian rings.

  (i) A subset $B$ of a ring $A$ is \emph{weakly closed} if for each
  pair of elements $a, b \in B$ there exists an element $\gamma(a,b)$
  in the center of $A$ such that $ab + \gamma(a,b) ba \in B$. In
  \cite{Jac} it is proved that if $A$ is artinian then subrings of $A$
  generated by nil weakly closed subsets are nilpotent. The earlier
  theorem of Levitski states that a nil one-sided ideal of a right
  noetherian ring is nilpotent.

  (ii) Goldie's Theorem can be employed to obtain the following
  corollary, also due to Goldie \cite[Theorem 6.1]{Gol}: If $A$ is a
  right noetherian ring, then a nil weakly closed subset of $A$
  generates a nilpotent subring. In fact, the proof of this last
  result more generally shows: Suppose that $A$ is a ring, that the
  prime radical $J$ of $A$ is nilpotent, and that $A/J$ embeds in a
  right artinian ring (e.g., $A/J$ is right Goldie). Then the weakly
  closed nil subsets of $A$ generate nilpotent subrings of $A$.

  (iii) Montgomery and Small apply Goldie's result to conclude that
  graded nil subrings of noetherian group graded rings must be
  nilpotent; see \cite[Corollary 1.2]{Mon-Sma}. Their key insight
  in this situation is that the set of homogeneous elements of a
  graded subring of a group graded ring is weakly
  closed. Consequently, their argument can be applied to show that if
  $A$ is a group graded ring with nilpotent prime radical $J$, and if
  $A/J$ embeds in a right artinian ring, then the graded nil subrings 
  of $A$ are nilpotent.

  (iv) Assuming that $A$ is a ring with right Krull dimension, then
  the prime radical $J$ of $A$ is nilpotent and every semiprime factor
  ring of $A$ is right Goldie; see for example \cite[Chapter
  6]{McC-Rob} for details.  Therefore, by (iii), if $A$ is a group
  graded ring with right Krull dimension then graded nil subrings of
  $A$ are nilpotent.

  (v) Assuming that $A$ is a finitely generated algebra, over a field
  $k$, satisfying a polynomial identity, then the prime radical of $A$
  is nilpotent and every semiprime factor of $A$ is a Goldie ring; see
  for example \cite[Chapter 13]{McC-Rob}. Therefore, again by (iii),
  if $A$ is a group graded ring satisfying a polynomial identity, then
  the graded nil subrings of $A$ are nilpotent.
\end{note}

We now present the main result. Recall $R$, $\grR$, and $\grRp$ from
(\ref{setup}).

\begin{thm} \label{main-result} If $R$ is algebraic over $k$ then the
  following conditions are equivalent: {\rm (i)} $R$ is finite
  dimensional. {\rm (ii)} $\grR$ is finite dimensional. {\rm (iii)}
  $\grRp$ is nilpotent. {\rm (iv)} $\grR$ is right noetherian. {\rm
    (v)} $\grR$ has right Krull dimension. {\rm (vi)} $\grR$ satisfies
  a polynomial identity. \end{thm}

\begin{proof} To start, it is easy to check that (i), (ii), and (iii)
  are equivalent, as already noted in (\ref{setup}.iii). Next, it
  is immediate that (ii) implies (iv), (v), and (vi). 

  That (iv) implies (iii) follows from (\ref{survey}.iii), that (v)
  implies (iii) follows from (\ref{survey}.iv), and that (vi) implies
  (iii) follows from (\ref{survey}.iv).

  The result follows.
\end{proof}

\begin{note} Following \cite[Chapter 8, \S 3]{McC-Rob} (also see
  \cite{Kra-Len}), if $A$ is a finitely generated algebra, over a
  field $k$, of finite GK-dimension, and if GK-dimension is
  \emph{finitely partitive} for $A$, then $A$ has finite right Krull
  dimension. Consequently, the finitely generated algebraic algebra
  $R$ of (\ref{main-result}) is finite dimensional if and only if
  $\grR$ is finitely partitive for GK-dimension. In \cite{Bel}, Bell
  gives examples of finitely generated algebras of Gelfand-Kirillov
  dimension 2, over arbitrary fields, that are not finitely partitive
  for GK-dimension.
\end{note}



\begin{thebibliography}{99}
  
\bibitem{Bel} J. P. Bell, Examples in finite Gel'fand-Kirillov
  dimension II, \emph{Comm.~Algebra}, 33 (2005), 3323--3334.

\bibitem{Bel-Sma} J. P. Bell and L. W. Small, A question of
  {K}aplansky, Special issue in celebration of Claudio Procesi's 60th
              birthday, \emph{J. Algebra}, 258 (2002), 386--388.

\bibitem{Bel-Sma-Smo} J. P. Bell, L. W. Small, A. Smoktunowicz,
  Primitive algebraic algebras of polynomially bounded growth, in
  \emph{New trends in noncommutative algebra}, Contemp.~Math.~562
  (Amer.~Math.~Soc., Providence, 2012), 41--52.

\bibitem{Bro-Goo} K. A. Brown and K. R. Goodearl, \emph{Lectures on 
    algebraic quantum groups}, Advanced Courses in Mathematics, CRM 
  Barcelona (Birkh\"auser Verlag, Basel, 2002). 

\bibitem{Gol} A. W. Goldie, Semi-prime rings with maximum condition,
  \emph{Proc.~London Math.~Soc.~(3)}, 10 (1960), 201--220.

\bibitem{Gol-Saf} E. S. Golod and I. R. \v Safarevi\v c, I. R., On the
  class field tower, \emph{Izv.~Akad.~Nauk SSSR Ser.~Mat.}, 28 (1964),
  261--272.

\bibitem{Jac} N. Jacobson, Une g\'en\'eralisation du th\'eor\`eme
  d'{E}ngel, \emph{C. R. Acad.~Sci.~Paris}, 234 (1952), 579--581.

\bibitem{Kra-Len} G\"unter R. Krause and T. H. Lenagan, \emph{Growth
    of algebras and {G}elfand-{K}irillov dimension}, Graduate Studies
  in Mathematics 22, Revised Edition (American Mathematical Society,
  Providence, RI, 2000).

\bibitem{Len-Smo} T. H. Lenagan and A. Smoktunowicz, An infinite
  dimensional affine nil algebra with finite {G}elfand-{K}irillov
  dimension, \emph{J. Amer.~Math.~Soc.}, 20 (2007), 989--1001.

\bibitem{Len-Smo-You} T. H. Lenagan, A. Smoktunowicz, and A. Young,
  Nil algebras with restricted growth,
  \emph{Proc.~Edinb.~Math.~Soc.~(2)}, 55 (2012), 461--475.

\bibitem{McC-Rob} J. C. McConnell and J. C. Robson, \emph{Noncommutative
    Noetherian Rings}, Graduate Studies in Mathematics 30 (American
  Mathematical Society, Providence, 2000).

\bibitem{Mon-Sma} S. Montgomery and L. W. Small, Nil subsets of graded
  algebras, \emph{Proc.~Amer.~Math.~Soc.}, 126 (1998), 653--656.

\bibitem{Nas-VOy} C. N\u ast\u asescu and F. van Oystaeyen,
  \emph{Graded ring theory}, North-Holland Mathematical Library
28 (North-Holland Publishing Co., Amsterdam-New York, 1982).

\bibitem{Reg} A. Regev, Filtered algebraic algebras,
  \emph{Proc.~Amer.~Math.~Soc.}, 138, (2010), 1941--1947.

\bibitem{Smo1} A. Smoktunowicz, A simple nil ring exists,
  \emph{Comm.~Algebra}, 30 (2002), 27--59.

\bibitem{Smo2} \bysame, Some results in noncommutative ring
  theory, in \emph{International {C}ongress of {M}athematicians.~{V}ol.~{II}},
(Eur.~Math.~Soc., Z\"urich, 2006), 259--269.

\bibitem{Smo3} \bysame, Graded algebras associated to algebraic
  algebras need not be algebraic, in \emph{European {C}ongress of
    {M}athematics} (Eur.~Math.~Soc., Z\"urich, 2010), 441--449.\

\bibitem{Zel} E. Zelmanov, Some open problems in the theory of infinite dimensional
              algebras, \emph{J. Korean Math.~Soc.}, 44 (2007), 1185--1195.



\end{thebibliography}
\end{document}